\documentclass[amssymb, 11pt]{amsart}
\usepackage{latexsym}

\newlength{\standardunitlength}
\newtheorem{prop}{Proposition}[section]

\newtheorem{lemma}[prop]{Lemma}
\newtheorem{cor}[prop]{Corollary}
\newtheorem{theorem}[prop]{Theorem}

\begin{document}

\title[Fixed points of non-uniform permutations]{Fixed points of non-uniform permutations and representation theory of the symmetric group}

\author{Jason Fulman}
\address{Department of Mathematics, University of Southern California, Los Angeles, CA 90089-2532, USA }
\email{fulman@usc.edu}


\subjclass[2020]{}

\date{June 27, 2024}

\begin{abstract} We use representation theory of the symmetric group $S_n$ to prove Poisson limit theorems for the
distribution of fixed points for three types of non-uniform permutations. First, we give results for the commutator $g^{-1}x^{-1}gx$
 where $g$ and $x$ are uniform in $S_n$. Second, we give results for the commutator $g^{-1}x^{-1}gx$ where $g$ in uniform in $S_n$ and $x$ is fixed.
Third, we give results for permutations obtained by multiplying $\frac{1}{i}n \log(n) + cn$  many random i-cycles. Some of our results are known by other, quite different, methods. \end{abstract}

\maketitle

\section{Introduction}

A classic result in the theory of the symmetric group $S_n$ is that the distribution of fixed points of a random permutation tends to a Poisson($1$) random
variable as $n \rightarrow \infty$. Here a random variable $X$ is said to have the Poisson($\lambda$) distribution if for all natural numbers $j$, \[ P(X=j) = \frac{\lambda^j}{j! e^{\lambda}}.\]There are dozens of proofs of this; see for instance \cite{AT} \cite{CDM} and the references therein. A more subtle problem is the study of the distribution of fixed points for non-uniform random permutations. All of the measures on permutations studied here are constant on conjugacy classes, which suggests that representation theory could be helpful. We will use the ``method of moments'' \cite{Diac}, so it is useful to know that a Poisson($\lambda$) random variable has rth moment equal to \[ \sum_{a=0}^r S(r,a) \lambda^a, \] where $S(r,a)$ is the Stirling number of the second kind, which is the number of set partitions of a set of size $r$ into $a$ blocks. In particular, the rth moment of a Poisson($1$) distribution is equal to the rth Bell number $B(r)$.
In fact for the uniform distribution on permutations, the first $n$ moments of the number of fixed points equal the first $n$ moments of the Poisson(1) distribution \cite{DS}.

Let us now describe in some more detail the three types of non-uniform permutations we study. 

First, we show that the number of fixed points of $g^{-1}x^{-1}gx$ where $x$ and $g$ are both uniform in $S_n$ tends to a Poisson(1) distribution
as $n \rightarrow \infty$. This result is due to Nica (Corollary 1.2 of \cite{N}) and later Linial and Puder (Theorem 25 of \cite{LP}). Neither of these papers
uses representation theory. A somewhat related result is Ore's conjecture, posed in 1951 \cite{O} and solved in 2010 \cite{LOST} following a long line of work on various cases, which states that every element of a finite non-abelian simple group $G$ is a commutator. This deep result relies on the fact, which we will also implicitly use, that the number of ways of writing an element $y$ as a commutator is equal to \[ |G| \sum_{\chi} \frac{\chi(g)}{\chi(1)}, \] where $\chi$ ranges over all irreducible characters of $G$.
Since our interest is in $G=S_n$, whose irreducible representations are parametrized by partitions of $n$, we will usually write our sums as over partitions of $n$.

Second, we study the distribution of fixed points of the commutator $g^{-1}x^{-1}gx$, where $g$ is uniform in $S_n$ and $x$ is fixed. Diaconis, Evans, and Graham \cite{DEG} studied this in depth in the case that $x$ is an $n$-cycle. They also proved that under weak conditions on the number of fixed points and two-cycles of $x$, that the total variation distance between the distribution of fixed points of $g^{-1}x^{-1}gx$ and a Poisson(1) distribution is small when $n$ is large. They use ``Stein's method'', and a remarkable aspect of their result is that they get an error term, although without explicit constants. Note that representation theory does not appear in the work, and the relevance of representation theory to the study of $g^{-1}x^{-1}gx$ with $x$ fixed follows from a known (but not widely known) ``functional equation'' for characters of finite groups. We do not obtain an error term for the Poisson approximation, although our formulas for the moments of the distribution might eventually be useful for that purpose.

Third, we study the distribution of fixed points of a permutation obtained by multiplying  $\frac{1}{i}n \log(n)+cn$ many i-cycles, where $c$ is a fixed real number and $i$ is fixed. We prove that the distribution of fixed points converges to a Poisson($1+e^{-ic}$) distribution as $n \rightarrow \infty$. This is a new result, previously known for $i=2$. The $i=2$ case was first proved by Matthews \cite{Mat} using strong uniform times. The $i=2$ case is also implicit in the recent breakthrough work of Teyssier \cite{Te}, who proved that for $c$ a fixed real number, the total variation distance between the distribution after $\frac{1}{2}n \log(n)+cn$  transpositions and the uniform distribution on $S_n$ converges to the total variation distance between a Poisson($1+e^{-2c}$) distribution and a Poisson(1) distribution. Teyssier's work does use representation theory, but our approach is quite different.

The mixing time for the random i-cycle Markov chain has been studied by Berestycki, Schramm, and Zeitouni \cite{BSZ} without representation theory, using clever probabilistic techniques, and by Hough \cite{Ho} using representation theory. In fact a modification of one of Hough's estimates on character ratios of i-cycles, due to Jimmy He and appearing here for the first time, will be very useful for us. 

A key ingredient in our work is a remarkable formula for the decomposition of the $r$th tensor power of the defining representation of the symmetric group.
This follows in a straightforward way from our earlier papers and seems to be new, which is a quite surprising fact. We also include a second proof, due to Alex Miller. As we point out, our formula implies a formula of Ding \cite{Din} which in turn implies a formula of Goupil and Chauve \cite{GC}. The formula of Ding requires that $1 \leq r \leq n-\lambda_2$, whereas our formula hods without restrictions. For our applications to commutators in this paper, our formula is more useful than Ding's. For our applications to the $i$-cycle walk, one could also use Ding's version. However even here our version could more prove useful in the future, since if one wants to obtain error terms for our theorems using moment generating functions, it is helpful to have control over all moments, including those growing with $n$.

The paper is organized as follows. Section \ref{prelim} gives useful formulas for the decomposition of the $r$th tensor power of the defining representation. Section \ref{commutators} treats our first two applications: the distribution of fixed points of the commutator $g^{-1}x^{-1}gx$ with $g,x$ random, and the distribution of fixed points of the commutator $g^{-1}x^{-1}gx$ with $g$ random and $x$ fixed. Section \ref{transposwalk} treats the distribution of fixed points of a permutation obtained by multiplying $\frac{1}{i}n \log(n)+cn$ random i-cycles, where $c$ and $i$ are fixed.

Throughout this paper we use standard terminology about Young tableaux and representation theory of the symmetric group; see \cite{Mac} or \cite{Sagan}. So $\lambda$ denotes a partition of size $|\lambda|$, $\lambda_j$ is the size of the jth part of $\lambda$, $d_{\lambda}$ is the number of standard Young tableaux of shape $\lambda$, and $d_{\lambda/\mu}$ denotes the number of standard Young tableaux of skew shape $\lambda/\mu$. We also remind the reader that all irreducible characters of $S_n$ are real valued.

\section{Decomposition of powers of the defining representation} \label{prelim}

In all that follows we let $m_{\lambda,r}$ denote the multiplicity of the representation corresponding to $\lambda$ in the $rth$ tensor power of the $n$-dimensional defining representation of the symmetric group. For example if $r=1$, then $m_{(n),1}=1$, $m_{(n-1,1)}=1$ and other $m_{\lambda,1}$ are equal to zero. Theorem \ref{explicit} is our main new tool and gives a formula for $m_{\lambda,r}$. This involves the Stirling number of the second kind, which is the number of set partitions of $r$ into $a$ blocks and is denoted $S(r,a)$.

\begin{theorem} \label{explicit} Let $d_{\lambda/\mu}$ denote the number of standard Young tableaux of skew shape $\lambda/\mu$. Then
\[ m_{\lambda,r} = \sum_{a=0}^r S(r,a) d_{\lambda/(n-a)}. \] 
\end{theorem}

\begin{proof} By Theorem 3.1 of \cite{F1}. $m_{\lambda,r} d_{\lambda}/n^r$ is equal to the probability that $\lambda$ is the RSK shape of a permutation obtained after $r$ top to random shuffles (see \cite{Sagan} for background on the RSK correspondence). So by Theorem 17 of \cite{F2}, it follows that
\begin{equation} \label{eq1}  m_{\lambda,r} = \frac{n^r}{n!} \sum_{a=0}^n P(a,r,n) (n-a)! d_{\lambda/(n-a)}, \end{equation} where $P(a,r,n)$ is the probability that when $r$ balls are dropped into $n$ cells, that there are $a$ occupied cells. Now we claim that \begin{equation} \label{eq2} P(a,r,n) = \frac{1}{n^r} \frac{n!}{(n-a)!} S(r,a). \end{equation} Indeed, there are $n^r$ many ways to drop $r$ balls into $n$ cells; if exactly $a$ cells are occupied, one chooses the $a$ cells in ${n \choose a}$ many ways and then puts the $r$ balls into the $a$ cells in $S(r,a) a!$ many ways. Combining \eqref{eq1} and \eqref{eq2} completes the proof.
\end{proof}

{\it Remark:} Our proof of Theorem \ref{explicit} used connections with card shuffling. We consider this to be a good thing, and in an earlier
version of this paper commented that a more direct proof would be desirable. Alex Miller (personal communication) contributed the following
proof which may help to explain the occurrence of Stirling numbers.

\begin{proof} (Second proof of Theorem \ref{explicit}) Let $U$ and $D$ be short for induction and restriction between $S_n$ and $S_{n-1}$ and let $1$ be the trivial character. The character we want to decompose is $(U1)^r$.

We claim that $(U1) \chi = UD \chi$ for any character $\chi$, so in particular $(U1)^r$ is equal to $(UD)^r(1)$. To prove the claim, note that by
Frobenius reciprocity \[ <U1 \cdot \chi, \psi> = <U1, \chi \psi> = <1,D\chi \cdot D\psi> = <D\chi,D\psi> = <UD\chi,\psi>.\] Since this is true
for all $\psi$, the claim follows.

It is known \cite{Standif} that $DU-UD$ is the identity operator and this allows one to move the $D$'s from left to right past the $U$'s using the rule
\[ DU^k = kU^{k-1} + U^kD.\] It follows by induction that \begin{equation} \label{stirling}  (UD)^r = \sum_{a=0}^r S(r,a) U^a D^a \end{equation} (this equation is also Proposition 4.9  of \cite{Standif}). By the branching rules for irreducible representations of the symmetric group,
the operator $D$ removes a corner box from a partition and the operator $U$ adds a corner box. So applying
\eqref{stirling} to the one row partition $(n)$ and taking the coefficient of $\lambda$ gives another proof of Theorem \ref{explicit}.
\end{proof}

{\it Remark:} Alex Miller has made the following observation. Theorem \ref{explicit} tells us how $(U1)^r$ decomposes:
\[ (U1)^r = \sum_{\lambda} m_{\lambda,r} \chi^{\lambda}.\] Evaluating both sides at the identity gives
\[ n^r = \sum_{\lambda} m_{\lambda,r} d_{\lambda}.\] This last identity appears in \cite{HL} and was proved using Schur-Weyl
duality. Their paper is well cited including a very recent paper by Krattenthaler \cite{Kr} concerning bijective proofs.
\[ \]
The following theorem of Ding \cite{Din} gives a formula for $m_{\lambda,r}$, assuming that $1 \leq r \leq n-\lambda_2$ (here $\lambda_2$ denotes the size of the second part of $\lambda$). We let $\bar{\lambda}$ denote the partition obtained from $\lambda$ be removing its largest part. The proof we give shows that Theorem \ref{dingsformula} is a straightforward consequence of Theorem \ref{explicit}.

\begin{theorem} \label{dingsformula} Suppose that $1 \leq r \leq n-\lambda_2$. Then
\[ m_{\lambda,r} = d_{\bar{\lambda}} \sum_{a=0}^r S(r,a) {a \choose |\bar{\lambda}|}.\]
\end{theorem} 

\begin{proof} Since $1 \leq r \leq n-\lambda_2$, we know that if $0 \leq a \leq r$, then $\lambda_2 \leq n-a$. It also follows that
\[ d_{\lambda/(n-a)} = d_{\bar{\lambda}} {a \choose |\bar{\lambda}|}.\] Indeed, to get $\lambda$ from $(n-a)$ by adding $a$ many boxes one at a time,
one must choose $a-n+\lambda_1$ steps out of $a$ steps to add a box to the first row (in  ${a \choose a-n+\lambda_1} = {a \choose \bar{\lambda}}$ many ways), and the other steps are used to construct $\bar{\lambda}$ in $d_{\bar{\lambda}}$ many ways.
\end{proof} 

\section{Commutators} \label{commutators}

In this section we study the distribution of the number of fixed points for commutators  $g^{-1}x^{-1}gx$ with $g,x$ random, and the distribution of the number of fixed points of the commutator $g^{-1}x^{-1}gx$ with $g$ random and $x$ fixed.

Lemmas \ref{firstlem} and \ref{seclem} are both known. However Lemma \ref{firstlem} is not well known (the only proof we found was on page 397 of \cite{Mac} and it's hard to follow), so we include a different proof and show that Lemma \ref{seclem} follows from it.

\begin{lemma} \label{firstlem} Let $\chi$ be an irreducible character of a finite group $G$. Then for all elements $x$ of $G$,
\[ \frac{1}{|G|} \sum_{g \in G} \chi(g^{-1}x^{-1}gx) = \frac{|\chi(x)|^2}{\chi(1)} .\]
\end{lemma}

\begin{proof} By replacing $x$ by $x^{-1}$ and $y$ by $x$, it is enough to show that
\[ \chi(x) \chi(y) = \frac{\chi(1)}{|G|} \sum_{g \in G} \chi(g^{-1}xgy) \] for all $x,y$. Indeed, $\chi(x^{-1})=\overline{\chi(x)}$.

Letting $\rho$ be the representation with character $\chi$, and letting $Tr$ denote trace,
one computes that
\begin{eqnarray*}
\sum_{g \in G} \chi(g^{-1}xgy) & = & \sum_{g \in G} Tr(\rho(g^{-1}xgy)) \\
& = & \sum_{g \in G} Tr(\rho(g^{-1}xg) \cdot \rho(y)) \\
& = & Tr(\sum_{g \in G} \rho(g^{-1}xg) \cdot \rho(y)).
\end{eqnarray*} 

One checks that $\sum_{g \in G} \rho(g^{-1}xg)$ commutes with all $\rho(y)$, so by Schur's lemma
(Corollary 1.6.8 of \cite{Sagan}), \[ \sum_{g \in G} \rho(g^{-1}xg) = c I \] is a multiple of the identity. Taking traces gives that
\[ c= \frac{|G| \chi(x)}{\chi(1)}.\] Thus
\[ \sum_{g \in G} \chi(g^{-1}xgy) = |G| \frac{\chi(x)}{\chi(1)} \chi(y).\]
\end{proof}

\begin{lemma} \label{seclem} Let $\chi$ be an irreducible character of a finite group $G$. Then
\[ \frac{1}{|G|^2} \sum_{x,g \in G} \chi(g^{-1}x^{-1}gx) = \frac{1}{\chi(1)}.\]
\end{lemma}

\begin{proof} Using Lemma \ref{firstlem} and then the orthogonality relation for irreducible characters gives
\[ \frac{1}{|G|} \sum_x \frac{1}{|G|} \sum_g \chi(g^{-1}x^{-1}gx) = \frac{1}{|G|} \sum_x \frac{|\chi(x)|^2}{\chi(1)}
= \frac{1}{\chi(1)}.\]
\end{proof}

{\it Remark:} There has been some use of symmetric function theory to study commutators. See for example Exercise 7.68 of \cite{Stan};
part g gives a formula for the expected number of j-cycles of $g^{-1}x^{-1}gx$ where $x,g$ are uniform in $S_n$.

Theorem \ref{mainfirst} is one of our main results. It is due to Nica \cite{N} and was proved by completely different methods.

\begin{theorem} \label{mainfirst} The distribution of the number of fixed points of the the commutator $g^{-1}x^{-1}gx$ where $x,g$ are uniform in $S_n$ converges to a Poisson(1) distribution as $n \rightarrow \infty$.
\end{theorem}

\begin{proof} By Lemma \ref{seclem}, the rth moment of the distribution of the number of fixed points of $g^{-1}x^{-1}gx$
where $x,g$ are uniform in $S_n$ is equal to \[ \sum_{|\lambda|=n} \frac{m_{\lambda,r}}{d_{\lambda}}.\] From Theorem
\ref{explicit} this is equal to 
\[ \sum_{\lambda} \sum_{a=0}^r S(r,a) \frac{d_{\lambda/(n-a)}}{d_{\lambda}}.\] Switching the order of summation gives
\[\sum_{a=0}^r S(r,a) \sum_{\lambda} \frac{d_{\lambda/(n-a)}}{d_{\lambda}}.\]

Now for $r$ fixed and large $n$, since $a \leq r$, we claim that 
\begin{equation} \label{mainclaim} \sum_{\lambda} \frac{d_{\lambda/(n-a)}}{d_{\lambda}} = 1 + O(1/n), \end{equation}
where the constant in $O(1/n)$ could depend on $r$. This would imply the theorem since
\[ \sum_{a=0}^r S(r,a) (1+O(1/n)) = B(r)+O(1/n), \] where $B(r)$ is the rth Bell number (the total number of set partitions of an
$r$ element set) and the rth moment of a Poisson(1) distribution is equal to $B(r)$.

To prove \eqref{mainclaim}, note that the partition $(n)$ contributes 1 to the sum. Since $a\leq r$, one has that $d_{\lambda/(n-a)}$ is upper bounded by
a constant depending on $r$ and if $\lambda$ is not equal to $(n)$ or $(1^n)$ and $n \geq 5$, then $d_{\lambda} \geq n-1$ by Appendix
C of \cite{Burn}. Also, the number of $\lambda$ containing $(n-a)$ is bounded by a constant depending on $r$. These observations complete
the proof of \eqref{mainclaim}.
\end{proof}

We also record the following corollary of the proof of Theorem \ref{mainfirst}.

\begin{cor} For all $n$ and $r$, the rth moment  of the distribution of the number of fixed points of $g^{-1}x^{-1}gx$, with $g,x$ uniform in $S_n$, is greater than or equal to $B(r)$, the rth moment of the Poisson(1) distribution.
\end{cor} 

\begin{proof} From the proof of Theorem \ref{mainfirst}, the rth moment is a sum of positive terms:
\[ \sum_{a=0}^r S(r,a) \sum_{\lambda} \frac{d_{\lambda/(n-a)}}{d_{\lambda}}.\] Summing only over $\lambda=(n)$ yields $B(r)$.
\end{proof}

Next we study the commutator $g^{-1}x^{-1}gx$ where $g$ is random and $x$ is fixed.

\begin{theorem} \label{complow} Consider $g^{-1}x^{-1}gx$ where $g$ is random and $x$ is a fixed element of $S_n$. Let $n_1(x)$ be the
number of fixed points of $x$ and let $n_2(x)$ be the number ot 2-cycles of $x$. 
\begin{enumerate}
\item The rth moment of the distribution of the number of fixed points of $g^{-1}x^{-1}gx$ is equal to
\[ \sum_{|\lambda|=n} m_{\lambda,r} \frac{\chi^{\lambda}(x)^2}{d_{\lambda}}.\] 
\item The mean of the distribution of the number of fixed points of  $g^{-1}x^{-1}gx$ is equal to
\[ 1 + \frac{(n_1(x)-1)^2}{n-1}.\]
\item The second moment of the distribution of the number of fixed points of  $g^{-1}x^{-1}gx$ is equal to

\begin{eqnarray*}
& & 2 + \frac{3(n_1(x)-1)^2}{n-1} + \frac{[{n_1(x)-1 \choose 2} + n_2(x)-1]^2}{n(n-3)/2} \\
& & + \frac{[{n_1(x)-1 \choose 2} - n_2(x)]^2}{(n-1)(n-2)/2}.
\end{eqnarray*}
\end{enumerate}
\end{theorem}

\begin{proof} The first assertion is clear from Lemma \ref{firstlem}.

For the second assertion, we know from Lemma \ref{explicit} that $m_{(n),1}=m_{(n-1,1),1}=1$ and $m_{\lambda,1}=0$ otherwise.
Thus by part 1, the first moment of the distribution of the number of fixed points is equal to
\[ \frac{\chi^{(n)}(x)^2}{d_{(n)}} + \frac{\chi^{(n-1,1)}(x)^2}{d_{(n-1,1)}} .\] Since $\chi^{(n)}(x)=1$ and $\chi^{(n-1,1)}(x)=n_1(x)-1$,
the result follows.

For the third assertion, we know from Theorem \ref{explicit} that $m_{(n),2}=2$, $m_{(n-1,1),2}=3$, $m_{(n-2,2),2}=1$,
$m_{(n-2,1,1),2}=1$ and $m_{\lambda,2}=0$ otherwise. Thus by part 1, we know that the second moment of the distribution of
the number of fixed points is equal to
\[ 2\frac{\chi^{(n)}(x)^2}{d_{(n)}} + 3 \frac{\chi^{(n-1,1)}(x)^2}{d_{(n-1,1)}} + \frac{\chi^{(n-2,2)}(x)^2}{d_{(n-2,2)}} 
+ \frac{\chi^{(n-2,1,1)}(x)^2}{d_{(n-2,1,1)}}.\] Since (see page 51 of \cite{FulHar})
\[ \chi^{(n-2,2)}(x) = {n_1(x)-1 \choose 2} + n_2(x) -1 \] and
\[ \chi^{(n-2,1,1)}(x) = {n_1(x)-1 \choose 2} - n_2(x), \] the result follows.
\end{proof}

{\it Remark:} Section 6 of Diaconis, Evans, and Graham \cite{DEG} proves part 2 of Theorem \ref{complow} (so computing the
mean) without using representation theory. Parts 1 and 3 seem to be new.

{\bf Example 1:} The lead example of \cite{DEG} is the case where $x$ is an $n$-cycle. Then $|\chi^{\lambda}(x)|=1$ if $\lambda=(n-j,1^j)$ is a hook
and $|\chi^{\lambda}(x)|=0$ otherwise. So by part 1 of Theorem \ref{complow}, the rth moment of the distribution of the number of fixed points of
$g^{-1}x^{-1}gx$ is equal to \[ \sum_{\lambda \  hook \atop |\lambda|=n} \frac{m_{\lambda,r}}{d_{\lambda}}.\] From Theorem
\ref{explicit}, this is equal to 
\begin{equation} \label{start} \sum_{a=0}^r S(r,a) \sum_{\lambda \ hook} \frac{d_{\lambda/(n-a)}}{d_{\lambda}}.
\end{equation}

Now looking at only the $\lambda=(n)$ term shows that \eqref{start} is at least $B(r)$. And clearly \eqref{start} is at most
\[ \sum_{a=0}^r S(r,a) \sum_{\lambda} \frac{d_{\lambda/(n-a)}}{d_{\lambda}}, \] where the sum is over all $\lambda$
of size $n$. This is at most $B(r)+O(1/n)$ by the calculation in the proof of Theorem \ref{mainfirst}.

It follows that if $x$ is an n-cycle, then the distribution of the number of fixed points of $g^{-1}x^{-1}gx$ converges to a Poisson($1$)
random variable as $n \rightarrow \infty$. Moreover, for any reasonable metric, 
the distribution of fixed points for $g^{-1}x^{-1}gx$ with $x$ an $n$-cycle should
be closer to Poisson($1$) than the distribution of fixed points for  $g^{-1}x^{-1}gx$ with $x,g$ both random is to a Poisson($1$) since
all of its moments are closer to $B(r)$.

{\bf Example 2:} Suppose that $x$ has cycle type $(k^{n/k})$ where $k \geq 3$ is fixed.

By Theorem \ref{complow}, the rth moment of the distribution of the number of fixed points of  $g^{-1}x^{-1}gx$ is
\[ \sum_{|\lambda|=n} m_{\lambda,r} \frac{\chi^{\lambda}(x)^2}{d_{\lambda}}.\] By Theorem \ref{explicit}, this is equal to
\[ \sum_{a=0}^r S(r,a) \sum_{\lambda} \frac{d_{\lambda/(n-a)}}{d_{\lambda}} \chi^{\lambda}(x)^2.\] So to prove that the
distribution of the number of fixed points tends to a Poisson($1$) limit as $n \rightarrow \infty$, it is enough to show that
\begin{equation} \label{labor} \sum_{\lambda \neq (n)} \frac{d_{\lambda/(n-a)}}{d_{\lambda}} \chi^{\lambda}(x)^2
\end{equation} tends to $0$ as $n \rightarrow \infty$.

From the formulas in the proof of part 3 of Theorem \ref{complow},
\[ \chi^{(n-1,1)}(x) = -1 \ , \ \chi^{(n-2,2)}(x) = -1 \ , \ \chi^{(n-2,1,1)}(x) = 0.\] So the $\lambda=(n-1,1)$, $\lambda=(n-2,2)$,
$\lambda = (n-2,1,1)$ terms in \eqref{labor} go to $0$ as $n \rightarrow \infty$. Since $a \leq r$ is fixed, the number of other terms
in \eqref{labor} is bounded by a constant depending only on $r$. Moreover each of the other terms has $d_{\lambda}$ of order at
least $n^3$. By Fomin and Lulov \cite{FominLulov}, \[ |\chi^{\lambda}(x)| = O \left( n^{(k-1)/2k} d_{\lambda}^{1/k} \right),\] so
\[ \frac{|\chi^{\lambda}(x)|^2}{d_{\lambda}} = O \left( \frac{n^{1-1/k}}{d_{\lambda}^{1-2/k}} \right) = 
O \left( \frac{n^{1-1/k}}{n^{3(1-2/k)}} \right) = O \left( \frac{n^{5/k}}{n^2} \right), \] which tends to $0$ as $n \rightarrow \infty$
since $k \geq 3$. This is consistent with the findings of \cite{DEG}; see Example 4 below.

{\bf Example 3:} Suppose that $x$ has cycle type $(2^{(n/2)})$. From Section 6 of \cite{DEG} the distribution of one-half multiplied by the
number of fixed points of $g^{-1}x^{-1}gx$ goes to a Poisson($1/2$) distribution as $n \rightarrow \infty$. Since a Poisson($\lambda$)
distribution has rth moment $\sum_{a=0}^r S(r,a) \lambda^a$, it follows that the rth moment of the distribution of fixed points
of $g^{-1}x^{-1}gx$ tends to
\begin{equation} \label{half}
2^r \sum_{a=0}^r S(r,a) \frac{1}{2^a}. \end{equation} It would be interesting to see this from the current perspective. Despite the
appearance of $S(r,a)$ in both \eqref{half} and Theorem \ref{explicit}, this does not seem completely straightforward.

{\bf Example 4:} Suppose that $a,b > 0$. Then Diaconis, Evans, and Graham \cite{DEG} show by Stein's method, that there exists
an (inexplicit) constant $K$ depending on $a$ and $b$ but not on $n$, such that if $g$ is uniformly distributed on $S_n$ and $x$ has at most
$a$ fixed points and at most $b$ 2-cycles, then the total variation distance between the distribution of the number of fixed points
of $g^{-1}x^{-1}gx$ and a Poisson$(1)$ distribution is at most $K/n$. It would be very interesting to prove this using moment generating
functions and the current perspective.

\section{Fixed points and the random i-cycle walk} \label{transposwalk}

This section studies the distribution of the number of fixed points after multiplying $\frac{1}{i} n \log(n) + cn$ many i-cycles, where $c$ is a fixed real number
and $n \rightarrow \infty$ and $i$ is fixed.

Lemma \ref{Fourier} is well known; see for instance \cite{Dgroup} or Exercise 7.67 of \cite{Stan}. 

\begin{lemma} \label{Fourier} The chance of obtaining a permutation $g$ after multiplying $k$ random i-cycles is equal to
\[ \frac{1}{n!} \sum_{|\tau|=n} d_{\tau}^2 \left( \frac{\chi^{\tau}(i,1^{(n-i)})}{d_{\tau}}  \right)^k \frac{\chi^{\tau}(g)}{d_{\tau}},\]
where the sum is over all partitions $\tau$ of $n$, and $\chi^{\tau}(i,1^{(n-i)})$ is the character of $\tau$ on an i-cycle.
\end{lemma}

This leads to the following corollary.

\begin{cor} \label{computer} The rth moment of the number of fixed points after multiplying $k$ many i-cycles is equal to
\[ \sum_{|\lambda|=n} d_{\lambda} \left( \frac{\chi^{\lambda}(i,1^{(n-i)})}{d_{\lambda}} \right)^k m_{\lambda,r}.\]
\end{cor}

\begin{proof} From Lemma \ref{Fourier}, the rth moment is equal to
\begin{equation} \label{complic} \sum_{g \in S_n}  \frac{1}{n!} \sum_{|\tau|=n} d_{\tau}^2 \left( \frac{\chi^{\tau}(i,1^{(n-i)})}{d_{\tau}}  \right)^k \frac{\chi^{\tau}(g)}{d_{\tau}} \sum_{|\lambda|=n} m_{\lambda,r} \chi^{\lambda}(g). \end{equation} By the orthogonality relation for irreducible characters,
\[ \frac{1}{n!} \sum_{g \in S_n} \chi^{\tau}(g) \chi^{\lambda}(g) = 1 \] if $\tau = \lambda$ and $0$ else. Thus \eqref{complic} simplifies to
\[ \sum_{|\lambda|=n} d_{\lambda} \left( \frac{\chi^{\lambda}(i,1^{(n-i)})}{d_{\lambda}} \right)^k m_{\lambda,r}.\]
\end{proof}

Lemma \ref{Jimmy} is due to Jimmy He and is included with his permission. It is a modification of Theorem 5 of Hough \cite{Ho}.
We had conjectured this lemma in an earlier version of the paper, but could only prove it for $i \leq 7$ using results from \cite{CGS}.
It would be interesting to prove Lemma \ref{Jimmy} without using complex analysis; perhaps the approaches of \cite{CGS} or \cite{IO}
will enable one to do this, but there are obstacles.
 
\begin{lemma} \label{Jimmy}
Fix natural numbers $t,i$, with $i\geq 2$. Let $\lambda$ be a partition of $n$ with $\lambda_1=n-t$. Then
\begin{equation*}
    \frac{\chi^\lambda(i,1^{(n-i)})}{d_\lambda}=1-\frac{it}{n}+O\left(\frac{1}{n^2}\right),
\end{equation*}
where the constant depends only on $t$ and $i$.
\end{lemma}

\begin{proof} The proof will use some notation. If $\lambda$ is a partition, the \emph{diagonal size} $m$ of $\lambda$ is the number of boxes on the diagonal, i.e. the largest $k$ such that $\lambda_k\geq k$. The \emph{Frobenius coordinates} $(a_1,\dotsc, a_m|b_1,\dots, b_m)$ of $\lambda$ are given by $a_j=\lambda_j-j$ and $b_j=\lambda'_j-j$. We let $x^{\underline{k}}=x(x-1) \cdots (x-k+1)$ denote the falling factorial.

 A formula of Frobenius (see page 52 of \cite{FulHar}) gives the value of the character ratio on an $i$-cycle as
\begin{equation*}
    \frac{\chi^\lambda(i,1^{(n-i)})}{d_\lambda}=-\frac{1}{in^{\underline{i}}}\oint F^{a,b}_i(z)\frac{dz}{2\pi \sqrt{-1}},
\end{equation*}
where $(a|b)$ are the Frobenius coordinates of $\lambda$, $m$ is the diagonal size of $\lambda$, and
\begin{equation*}
    F^{a,b}_i(z)=\left(z+\frac{i-1}{2}\right)^{\underline{i}}\prod_{j=1}^m \frac{z-a_j-\frac{i+1}{2}}{z-a_j+\frac{i-1}{2}}\prod_{j=1}^m\frac{z+b_j+\frac{i+1}{2}}{z+b_j-\frac{i-1}{2}}.
\end{equation*}
Here, the contour goes counterclockwise around all poles of the integrand. 

The key idea to prove the lemma, which already appeared in \cite{Ho}, is that the dominant contribution to the asymptotics comes from the residue of the pole at $a_1-\frac{i-1}{2}$, which is far from the origin. The remaining poles are all of constant order distance from the origin, and we can estimate their combined contribution.

 First, note that as $a_1=n-t-1$, and the remaining $a_j$ and $b_j$ are at most $t$, and so $O(1)$, for large enough $n$ we can choose a large constant $R$ such that $C_R$, the circle of radius $C$ centered at $0$, contains all poles except $a_1-\frac{i-1}{2}$. We can then write
\begin{equation}
\label{eq: res+int}
    \frac{\chi^\lambda(i,1^{(n-i)})}{d_\lambda}=-\frac{1}{in^{\underline{i}}}\mathrm{res}_{z=a_1-\frac{i-1}{2}}F_i^{a,b}(z)-\frac{1}{in^{\underline{i}}}\oint_{C_R} F^{a,b}_i(z)\frac{dz}{2\pi \sqrt{-1}}.
\end{equation}

Now the residue at $z=a_1-\frac{i-1}{2}$ was already computed in \cite{Ho}, and is easily seen to be
\begin{equation*}
    \mathrm{res}_{z=a_1-\frac{i-1}{2}}F_i^{a,b}(z)=-i(n-t-1)^{\underline{i}}\prod_{j=2}^m\frac{n-t-a_j-i-1}{n-t-a_j-1}\prod_{j=1}^m\frac{n-t+b_j}{n-t+b_j-i}.
\end{equation*}
and so the first term in \eqref{eq: res+int} is equal to
\begin{equation*}
    \frac{(n-t-1)^{\underline{i}}}{n^{\underline{i}}}\prod_{j=2}^m\frac{n-t-a_j-i-1}{n-t-a_j-1}\prod_{j=1}^m\frac{n-t+b_j}{n-t+b_j-i}.
\end{equation*}
We claim that this is $1-\frac{it}{n}+O\left(\frac{1}{n^2}\right)$. Indeed, the first factor is
\begin{equation*}
    \left(1-\frac{t+1}{n}\right)\dotsm \left(1-\frac{t+1}{n-i+1}\right)=1-\frac{i(t+1)}{n}+O\left(\frac{1}{n^2}\right),
\end{equation*}
the second factor is
\begin{equation*}
    \prod_{j=2}^m\left(1-\frac{i}{n-t-a_j-1}\right)=1-\frac{i(m-1)}{n}+O\left(\frac{1}{n^2}\right),
\end{equation*}
and the third factor is
\begin{equation*}
    \prod_{j=1}^m\left(1+\frac{i}{n-t+b_j}\right)=1+\frac{im}{n}+O\left(\frac{1}{n^2}\right).
\end{equation*}
Here, it's easy to see that all implied constants depend only on $i$ and $t$, as in particular $m\leq t$ and all $a_j$ and $b_j$ (except $a_1$) are also bounded by $t$. Thus, the first term in \eqref{eq: res+int} is $1-\frac{it}{n}+O\left(\frac{1}{n^2}\right)$.

It remains to bound the second term (the contour integral) in \eqref{eq: res+int}. To do so, we simply note that if we pick $R$ large enough, then $F_i^{a,b}(z)$ is of constant order along the contour $C_R$. Indeed, if say $t+i+1<R<\sqrt{n}$, then all denominators are at least $1$, all numerators except the one containing $a_1$ are bounded by $2R$, and $\left|\frac{z-a_1-\frac{i+1}{2}}{z-a_1+\frac{i-1}{2}}\right|=1+o(1)$, and so $|F_i^{a,b}(z)|\leq (2R)^{i+2m-1}(1+o(1))=O(1)$ with dependence only on $i$ and $t$, since $m\leq t$ and $R$ can be chosen to only depend on these two quantities. Since the contour is also of constant length, the contour integral is of constant order, and so the second term in \eqref{eq: res+int} is $O(n^{-i})$.
\end{proof}

Lemma \ref{intasy} is important. Recall that $\bar{\lambda}$ denotes the partition obtained from $\lambda$ by removing its largest part.

\begin{lemma} \label{intasy} If $\lambda$ is any partition satisfying $\lambda_1=n-t$ where $t$ is fixed, then for fixed $c$,
\[ d_{\lambda} \left( \frac{\chi^{\lambda}(i,1^{(n-i)})}{d_{\lambda}} \right)^{\frac{1}{i} n \log(n) +cn} \rightarrow
\frac{e^{-itc} d_{\bar{\lambda}}}{t!} \] as $n \rightarrow \infty$. \end{lemma}

\begin{proof} Multiplying and dividing by $n^t$ gives \[ \frac{d_{\lambda}}{n^t} n^t  \left( \frac{\chi^{\lambda}(i,1^{(n-i)})}{d_{\lambda}} \right)^{\frac{1}{i} n \log(n) +cn}.\] Now by the hook formula for $d_{\lambda}$ (see Theorem 3.1.2 of \cite{Sagan}), \[ \frac{d_{\lambda}}{n^t} \rightarrow \frac{d_{\bar{\lambda}}}{t!}.\] By Lemma \ref{Jimmy},
\[ \frac{\chi^{\lambda}(i,1^{(n-i)})}{d_{\lambda}} = 1 - \frac{it}{n} + O(1/n^2).\]
Thus \[ n^t  \left( \frac{\chi^{\lambda}(i,1^{(n-i)})}{d_{\lambda}} \right)^{\frac{1}{i} n \log(n) +cn} \rightarrow e^{-itc} \] as $n \rightarrow
\infty$. \end{proof}

Theorem \ref{maintrans} is the main result of this section. For $i=2$ it is due to Matthews \cite{Mat}, who proved it using strong uniform times.
We have not seen it in the literature for $i \geq 3$.

\begin{theorem} \label{maintrans} For $c$ a fixed real number and $i \geq 2$ fixed, the distribution of the number of fixed points after multiplying
$\frac{1}{i} n \log(n) +cn$ random i-cycles converges to a Poisson$(1+e^{-ic})$ limit as $n \rightarrow \infty$.
\end{theorem}  

\begin{proof} By Corollary \ref{computer} and Theorem \ref{explicit}, the rth moment of the number of fixed points after $k$
many i-cycles is equal to 
\[ \sum_{a=0}^r S(r,a) \sum_{|\lambda|=n} d_{\lambda} d_{\lambda/(n-a)} \left( \frac{\chi^{\lambda}(i,1^{(n-i)})}{d_{\lambda}} \right)^k .\]
Since a Poisson($\lambda$) distribution has rth moment $\sum_{a=0}^r S(r,a) \lambda^a$, it is sufficient to show that as $n \rightarrow \infty$,
\begin{equation} \label{eqstar} \sum_{|\lambda|=n} d_{\lambda} d_{\lambda/(n-a)}  \left( \frac{\chi^{\lambda}(i,1^{(n-i)})}{d_{\lambda}} \right)^{\frac{1}{i} n \log(n) +cn} \end{equation} tends to $(1+e^{-ic})^a$. 

Rewrite \eqref{eqstar} as \begin{equation} \label{eqstar2} \sum_{t=0}^a \sum_{|\lambda|=n \atop \lambda_1=n-t} d_{\lambda} d_{\lambda/(n-a)} 
 \left( \frac{\chi^{\lambda}(i,1^{(n-i)})}{d_{\lambda}} \right)^{\frac{1}{i} n \log(n) +cn}.\end{equation}
Arguing as in the proof of Theorem \ref{dingsformula}, for $a$ fixed and large enough $n$, \[ d_{\lambda/(n-a)} = d_{\bar{\lambda}} {a \choose t}.\] 
Thus for fixed $a$ and $n$ large enough, \eqref{eqstar2} becomes
\[ \sum_{t=0}^a {a \choose t} \sum_{|\lambda|=n \atop \lambda_1=n-t} d_{\bar{\lambda}} d_{\lambda}
\left( \frac{\chi^{\lambda}(i,1^{(n-i)})}{d_{\lambda}} \right)^{\frac{1}{i} n \log(n) +cn}.\] By Lemma \ref{intasy}, this tends to \begin{equation} \label{lastone} \sum_{t=0}^a e^{-itc} {a \choose t} \sum_{|\bar{\lambda}|=t} (d_{\bar{\lambda}})^2 / t!. \end{equation} as $n \rightarrow \infty$. Since $\sum_{|\bar{\lambda}|=t} (d_{\bar{\lambda}})^2 = t!$, \eqref{lastone} simplifies to
\[ \sum_{t=0}^a e^{-itc} {a \choose t} = (1+e^{-ic})^a, \] as needed.
\end{proof}

\section{Acknowledgements} The author was supported by Simons Foundation Grant 917224. Persi Diaconis helped with the exposition, Alex Miller
provided a second proof of Theorem 2.1 and made other comments, and Jimmy He proved the crucial Lemma 4.3 which was conjectured in an earlier version of this paper.

\end{document}